\documentclass[12pt,oneside]{amsart}

\usepackage{amssymb,latexsym}

\usepackage{hyperref}
\hypersetup{
  colorlinks   = true, 
  urlcolor     = blue, 
  linkcolor    = blue, 
  citecolor   = red 
}
\usepackage{anysize}
\usepackage{subcaption}

\usepackage{pgfplots}
\usepackage{mathrsfs}
\usetikzlibrary{arrows}
\usetikzlibrary{calc}
\usetikzlibrary{intersections}
\usetikzlibrary{through}

\usepackage{mathtools}

\newtheorem{theorem}{Theorem}[section]

\newtheorem{lemma}[theorem]{Lemma}

\theoremstyle{definition}

\theoremstyle{remark}

\numberwithin{equation}{section}

\renewcommand{\epsilon}{\varepsilon}
\renewcommand{\phi}{\varphi}
\renewcommand{\kappa}{\varkappa}

\begin{document}

\title{Note on illuminating constant width bodies}

\author{Alexey Glazyrin{$^\spadesuit$}}

\thanks{{$^\spadesuit$} Partially supported by the NSF grant DMS-2054536}

\address{Alexey Glazyrin, School of Mathematical \& Statistical Sciences, The University of Texas Rio Grande Valley, Brownsville, TX 78520, USA}
\email{alexey.glazyrin@utrgv.edu}


\begin{abstract}
Recently, Arman, Bondarenko, and Prymak constructed a constant width body in $\mathbb{R}^n$ whose illumination number is exponential in $n$. In this note, we improve their bound by generalizing the construction. In particular, we construct a constant width body in $\mathbb{R}^n$ whose illumination number is at least $(\tau+o(1))^n$, where $\tau\approx 1.047$.
\end{abstract}

\maketitle

\section{Introduction}

A boundary point $x$ of a convex $n$-dimensional body $B$ is said to be illuminated by a direction (unit vector) $\ell$ if the ray from $x$ in the direction of $\ell$ intersects the interior of $B$. The set of directions $\mathcal{L}$ illuminate $B$ if each boundary point of $B$ is illuminated by some direction from $\mathcal{L}$. A natural question is to determine the minimal size of an illuminating set, also called \textit{an illumination number} \cite{bol60}, for a given $B$ or for a given class of convex $n$-dimensional bodies.

In \cite{sch88}, Schramm showed that the illumination number of any $n$-dimensional body of constant width is no greater than $(\sqrt{3/2}+o(1))^n$. The question of existence of bodies of constant width with exponential illumination numbers (see \cite[Problem 3.3]{kal15}) was recently answered in the affirmative by Arman, Bondarenko, and Prymak \cite{arm23}. They constructed a body of constant width whose illumination number is at least $(\cos\pi/14 + o(1))^{-n}$.

The construction of Arman, Bondarenko, and Prymak is based on the union of congruent right spherical cones inscribed in the unit sphere. Cones are chosen in such a manner that the diameter of the union is equal to the diameter of each cone. Then there exists a body of constant width of the same diameter containing this union of cones \cite[Theorem 54, p. 126]{egg58}. Choosing the apexes of cones according to the economical covering of the sphere constructed by B{\"o}r{\"o}czky and Wintsche \cite{bor03} and estimating the number of apexes that can be illuminated by the same direction, Arman, Bondarenko, and Prymak show that there is a body of constant width with the illuminating number exponential in dimension. Their construction allows for the base of the exponential bound equal to $(\cos(\pi/14))^{-1}\approx 1.026$.

The main idea of this note is to generalize their construction and, by this, gain more freedom in choosing apexes of cones. For our construction, we consider right spherical cones whose apexes lie in the unit sphere but whose bases belong to a concentric sphere of a possibly different radius $R$. We fix the radius $R$, the distance from the apex to the base $d$, the angle $\alpha$ between the axis and the generatrix of a cone, and the spherical radius $\beta$ of the base sphere, see Figure \ref{fig:cone}.

\begin{figure}
\centering
\begin{subfigure}{.5\textwidth}
\centering
\includegraphics[width=.8\linewidth]{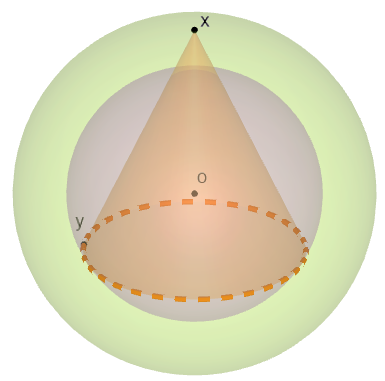}
\end{subfigure}%
\begin{subfigure}{.5\textwidth}
\centering
\includegraphics[width=.8\linewidth]{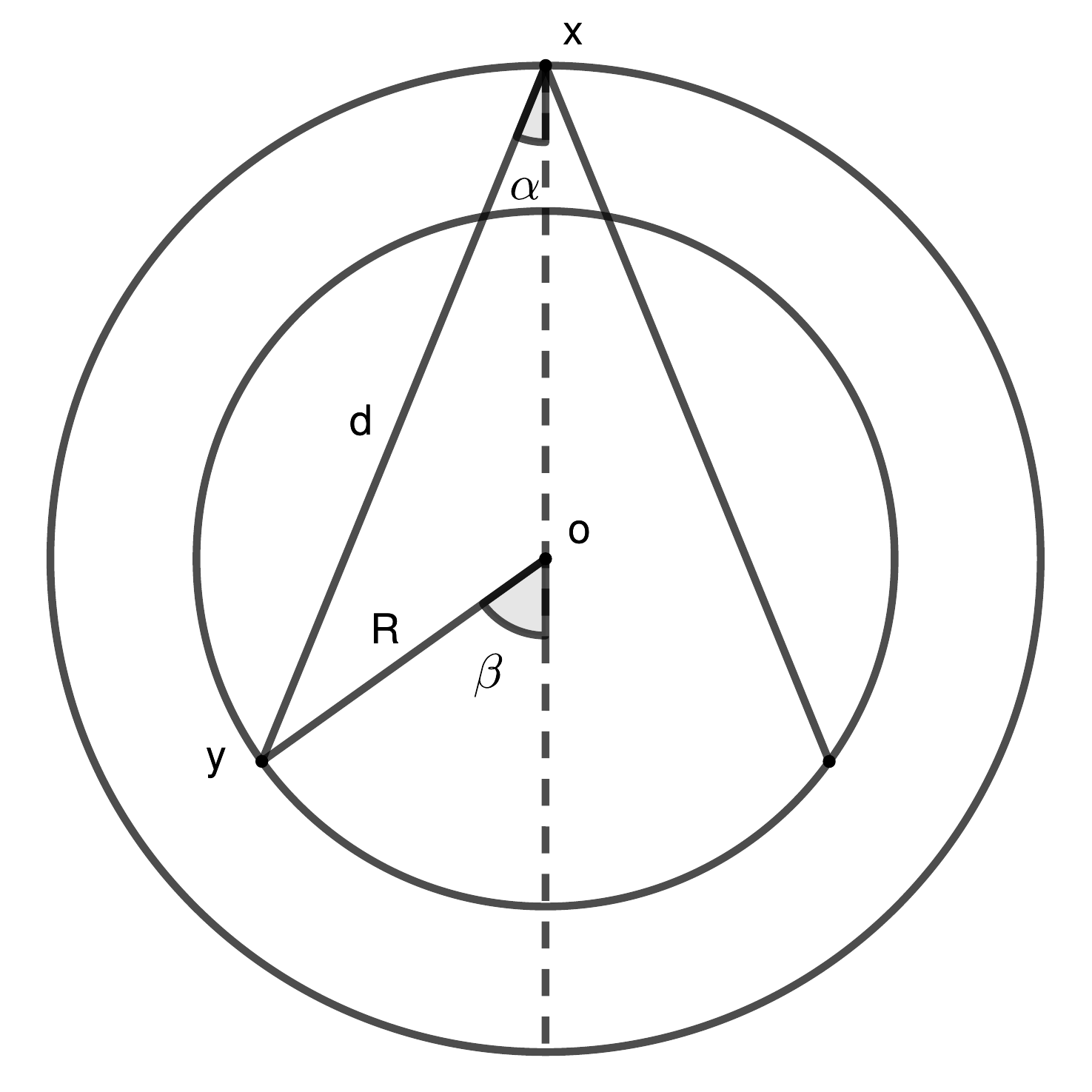}
\end{subfigure}
\caption{Cone and its axial section}
\label{fig:cone}
\end{figure}

These parameters are not independent but a separate notation for each of them will be more convenient for us. We always consider only cones whose diameter is $d$, that is, the condition $2R\sin\beta\leq d$ is always satisfied.

For each $x\in\mathbb{S}^{n-1}$, we denote by $Q(x)$ the cone with the fixed parameters $R, d, \alpha, \beta$ whose apex is $x$. For a subset $X\subset \mathbb{S}^{n-1}$, we denote the union $\bigcup\limits_{x\in X} Q(x)$ by $\mathcal{W}(X)$.

For $x,y\in\mathbb{S}^{n-1}$ the spherical distance between $x$ and $y$ is denoted by $\theta(x,y)$. For $0<\alpha<\pi/2$, the spherical cap with center $x$ and spherical radius $\alpha$ is denoted by $C(x,\alpha)$.

The following lemma describes sufficient conditions for the diameter of $\mathcal{W}(X)$ to be $d$, thus extending Lemma 3 in \cite{arm23}.

\begin{lemma}\label{lem:diam}
The diameter of $\mathcal{W}(X)$ is $d$ if the following conditions are satisfied for each pair of points $x,y\in X$:

\begin{enumerate}

\item\label{cond:1} $2\sin(\theta(x,y)/2) \leq d$;

\item\label{cond:2} $\theta(x,y)\geq 2\beta$;

\item\label{cond:3} $2R\leq d$ or $\theta(x,y)\leq 2 \arcsin \frac d {2R} - 2\beta$.

\end{enumerate}
\end{lemma}

For the main result of the note we also use two lemmas from \cite{arm23}.

\begin{lemma}\label{lem:illum}\cite[Lemma 1]{arm23}
Suppose $0<\alpha\leq\pi/6$. $K$ is a convex body in $\mathbb{R}^n$ whose diameter is $d$. For some $x\in\mathbb{S}^{n-1}$, $Q(x)\subset K$ and $x$ is on the boundary of $K$. Then $x$ is illuminated by $\ell$ only if $\ell\in C(-x,\pi/2 - \alpha)$.
\end{lemma}
Lemma \ref{lem:illum} is formulated in \cite{arm23} for cones inscribed in the unit sphere but the proof works in our case just as well.

\begin{lemma}\label{lem:prob}\cite[Lemma 2]{arm23}
Suppose $0<\psi<\phi<\pi/2$ are fixed. Then for every positive integer $n$ there exists $X\subset\mathbb{S}^{n-1}$ with $|X|=\left(\frac {1+o(1)} {\sin\phi} \right)^n$ such that $\psi \leq \theta(x,y) \leq \pi-\psi$ for each pair $x,y\in X$ and every point of $\mathbb{S}^{n-1}$ is contained in at most $O(n\log n)$ spherical caps $C(x,\phi)$, for all $x\in X$.
\end{lemma}

The main result of this note is the following theorem.

\begin{theorem}\label{thm:main}
For every positive integer $n$, there exists an $n$-dimensional body of constant width $K$ with illumination number at least $(\tau+o(1))^n$, where $\tau=\frac 1 4 \sqrt{\frac 1 6 (111-\sqrt{33})}\approx 1.047$.
\end{theorem}

\section{Proofs}

\begin{proof}[Proof of Lemma \ref{lem:diam}]

Due to condition (\ref{cond:1}), distances between $x,y\in X$ are no greater than $d$.

Bases of spherical cones form spherical caps of spherical radius $\beta$ on the concentric sphere of radius $R$. For $x\in \mathbb{S}^{n-1}$ the interior of the corresponding spherical cap formed by the base of $Q(x)$ consists precisely of all points on the sphere of radius $R$ whose distance to $x$ is greater than $d$. In order to ensure that distances from points in $X$ to all cone bases are no greater than $d$, we need for spherical caps to be non-intersecting. This means that the spherical distance between their centers, that is $\theta(x,y)$, is at least $2\beta$, which is precisely condition (\ref{cond:2}).

Finally, we need to check that the distances between points in two different cone bases are at least $d$. If $d\geq 2R$, this condition is satisfied automatically because $2R$ is the diameter of the sphere of radius $R$. If the second part of condition (\ref{cond:3}) is satisfied, then the spherical distance between two furthest points in spherical caps $\theta(x,y)+2\beta$ is no greater than $2\arcsin \frac d {2R}$ so the Euclidean distance between them is no greater than $d$.
\end{proof}

The proof of Theorem \ref{thm:main} follows the steps of the proof of the main result in \cite{arm23}. We just need to choose suitable parameters for the cone.

We set $d=2R$ so that condition (\ref{cond:3}) is automatically satisfied. We also impose the condition $2\beta+\alpha=\pi/2$ so that the angle $\alpha$ is maximal possible. By the law of cosines, $\cos\alpha=\frac {3R^2+1} {4R}$ and $\cos\beta=\frac {3R^2-1}{2R}$. Using these values, we solve the equation $\sin(2\beta)=\cos\alpha$ for $R$ and find the only feasible root $R_0=\frac 1 3 \sqrt{\frac 1 2 (9+\sqrt{33})}\approx 0.905$. Then $d_0=2R_0$, $\beta_0=\arccos\left(\frac 1 4 \sqrt{\frac 1 2 (15+\sqrt{33})} \right)$, $\alpha_0=\frac {\pi} 2 - 2\beta_0$.

\begin{proof}[Proof of Theorem \ref{thm:main}]
We set $\psi=2\beta_0$ and $\phi=\psi+\epsilon$, for a small $\epsilon>0$. Take the collection $X$ of points in $\mathbb{S}^{n-1}$ constructed by Lemma \ref{lem:prob}. For this collection, we take the union of corresponding cones $\mathcal{W}(X)$ with parameters $\alpha_0$, $\beta_0$, $R_0$, $d_0$ each.

We first check that the condition $\psi\leq \theta(x,y)\leq \pi-\psi$ for every pair of points $x,y\in X$ implies all conditions of Lemma \ref{lem:diam}. Indeed, condition $2R\sin\beta\leq d$ and conditions (\ref{cond:2}) and (\ref{cond:3}) are satisfied automatically. Condition $\theta(x,y)/2 \leq \pi/2 - \beta_0$ implies

$$2\sin(\theta(x,y)/2) - d_0 \leq 2\cos\beta_0 -d_0= 2\frac {3R_0^2-1}{2R_0} - 2R_0 = \frac {R_0^2-1} {R_0} \leq 0,$$
because $R_0<1$. Therefore, condition (\ref{cond:1}) is also satisfied and, by Lemma \ref{lem:diam}, the diameter of $\mathcal{W}(X)$ is $d$.

By \cite[Theorem 54, p. 126]{egg58}, there exists $K$, a body of constant width $d_0$ containing $\mathcal{W}(X)$. For each $x\in X$, $x$ is on the boundary of $K$ and, by Lemma \ref{lem:illum}, the set of directions illuminating it belongs to $C(-x,\pi/2 - \alpha_0)$. Due to the construction of the set $X$ via Lemma \ref{lem:prob}, each direction vector $\ell$ may belong only to $O(n\log n)$ spherical caps $C(-x,\phi)$. Since $\phi=\pi/2-\alpha_0+\epsilon$, each direction may illuminate only $O(n \log n)$ apexes from $X$. The illumination number of $K$ is, therefore, at least
$$\frac {|X|} {O(n\log n)}  = \left(\frac {1+o(1)} {\sin\phi} \right)^n = (\cos(\alpha_0-\epsilon)+o(1))^{-n}.$$
The theorem is then proved for $\tau=(\cos\alpha_0)^{-1}=\frac 1 4 \sqrt{\frac 1 6 (111-\sqrt{33})}\approx 1.047$
\end{proof}

\bibliographystyle{amsalpha}

\begin{thebibliography}{A}

\bibitem{arm23}
{\sc A. Arman, A. Bondarenko, and A. Prymak.}
\emph{Convex bodies of constant width with exponential illumination number}, \url{https://arxiv.org/abs/2304.10418}.

\bibitem{bol60}
{\sc V.G. Boltyanski}.
\emph{The problem of illumination of the boundary of a convex body (in Russian)}, Izv. Mold. Fil. Akad. Nauk SSSR, no. 10 (1960): 79-86.

\bibitem{bor03}
{\sc K.J. B{\"o}r{\"o}czky and G. Wintsche.}
\emph{Covering the sphere by equal spherical balls.}
In Discrete and Computational Geometry, The Goodman-Pollack Festschrift (2003): 235-251.

\bibitem{egg58}
{\sc H. G. Eggleston.}
\emph{Convexity}, Cambridge Tracts in Mathematics and Mathematical Physics, No. 47, Cambridge
University Press, New York, 1958.

\bibitem{kal15}
{\sc G. Kalai.}
\emph{Some old and new problems in combinatorial geometry I: around Borsuk’s problem}, Surveys in
combinatorics 2015, London Math. Soc. Lecture Note Ser., vol. 424, Cambridge Univ. Press, Cambridge,
2015, 147-174.

\bibitem{sch88}
{\sc O. Schramm.}
\emph{Illuminating sets of constant width}, Mathematika 35 (1988), no. 2, 180-189.

\end{thebibliography}

\end{document}